\documentclass[12pt]{amsart}
\usepackage{amsmath,amsthm,latexsym,amscd,amsbsy,amssymb,amsfonts,fleqn,leqno,
euscript, pb-diagram}

\numberwithin{equation}{section}

\newtheorem{thm}{Theorem}[section]
\newtheorem{pro}[thm]{Proposition}
\newtheorem{lem}[thm]{Lemma}
\newtheorem{cor}[thm]{Corollary}
\newtheorem{rem}[thm]{Remark}

\begin{document}


\title[Embeddings of finite-dimensional compacta in Euclidean spaces]
{Embeddings of finite-dimensional compacta in Euclidean spaces}

\author{S. Bogataya}
\address{Hight School of Economics, Moscow 119992,
Russia}\email{svetbog@mail.ru}

\author{S. Bogatyi}
\address{Faculty of Mechanics and Mathematics, Moscow State University, Vorob'evy
gory, Moscow 119899, Russia} \email{bogatyi@inbox.ru}
\thanks{The second author was supported by Grants NSH 1562.2008.1
and RFFI 09-01-00741-a.}

\author{V.  Valov}
\address{Department of Computer Science and Mathematics, Nipissing University,
100 College Drive, P.O. Box 5002, North Bay, ON, P1B 8L7, Canada}
\email{veskov@nipissingu.ca}
\thanks{The third author was supported by NSERC Grant 261914-08}




\begin{abstract}
If $g$ is a map from a space $X$ into $\mathbb R^m$ and $q$ is an
integer, let $B_{q,d,m}(g)$ be the set of all planes
$\Pi^d\subset\mathbb R^m$ such that $|g^{-1}(\Pi^d)|\geq q$. Let
also $\mathcal H(q,d,m,k)$ denote the maps $g\colon X\to\mathbb R^m$
such that $\dim B_{q,d,m}(g)\leq k$. We prove that for any
$n$-dimensional metric compactum $X$ each of the sets $\mathcal
H(3,1,m,3n+1-m)$ and $\mathcal H(2,1,m,2n)$ is dense and $G_\delta$
in the function space $C(X,\mathbb R^m)$ provided $m\geq 2n+1$ (in
this case $\mathcal H(3,1,m,3n+1-m)$ and $\mathcal H(2,1,m,2n)$ can
consist of embeddings). The same is true for the sets $\mathcal
H(1,d,m,n+d(m-d))\subset C(X,\mathbb R^m)$ if $m\geq n+d$, and
$\mathcal H(4,1,3,0)\subset C(X,\mathbb R^3)$ if $\dim X\leq 1$.
This results complements an authors' result from \cite{bv}. A
parametric version of the above theorem, as well as a partial answer
of a question from \cite{b2} and \cite{bv} are also provided.
\end{abstract}

\maketitle

\markboth{}{Embeddings in $\mathbb R^m$}



\section{Introduction}
In this paper we assume that all topological spaces are metrizable
and all single-valued maps are continuous.

Everywhere below by $M_{m,d}$ we denote the space of all affine
$d$-dimensional subspaces $\Pi^d$ (briefly, $d$-planes) of $\mathbb
R^m$. If $g$ is a map from a space $X$ into $\mathbb R^m$ and $q$ is
an integer, let $B_{q,d,m}(g)=\{\Pi^d\in M_{m,d}:|g^{-1}(\Pi^d)|\geq
q\}$. There is a metric topology on $M_{m,d}$, see \cite{dnf}, and
we consider $B_{q,d,m}(g)$ as a subspace of $M_{m,d}$ with this
topology. For a given space $X$ we consider the set $\mathcal
H(q,d,m,k)$ of all maps $g\colon X\to\mathbb R^m$ such that $\dim
B_{q,d,m}(g)\leq k$.

It follows from authors' result \cite[Corollary 1.6 with $T=m=2n+1$
and $t=0$]{bv} that if $X$ is metrizable compactum with $\dim X\leq
n$, then all maps $g\colon X\to\mathbb R^{2n+1}$ such that for every
$\Pi^1\in M_{2n+1,1}$ the preimage $g^{-1}(\Pi^1)$ contains at most
$4$ points form a dense and $G_\delta$-subset of $C(X,\mathbb
R^{2n+1})$ (here $C(X,\mathbb R^m)$ is the space of all maps from
$X$ into $\mathbb R^m$ with the uniform convergence topology). This
result can be complemented as follows:

\begin{thm}
Let $X$ be a metrizable compactum of dimension $\leq n$. Then:
\begin{itemize}
\item[(a)] The set $\mathcal H(3,1,m,3n+1-m)$ is dense and
$G_\delta$ in $C(X,\mathbb R^m)$ provided $m\geq 2n+1$.
\item[(b)] The set $\mathcal H(2,1,m,2n)$ is dense and
$G_\delta$ in $C(X,\mathbb R^m)$ provided $m\geq 2n+1$.
\item[(c)] The set $\mathcal H(1,d,m,n+d(m-d))$ is dense and
$G_\delta$ in $C(X,\mathbb R^m)$ provided $m\geq n+d$.
\item[(d)] The set $\mathcal H(4,1,3,0)$ is dense and
$G_\delta$ in $C(X,\mathbb R^3)$ if $n=1$.
\end{itemize}
\end{thm}

If $f\colon X\to Y$ is a perfect surjection, we denote by $\mathcal
P(q,d,m,k)$ the set of all maps $g\colon X\to\mathbb R^m$ such that
$\dim B_{q,d,m}(g|f^{-1}(y))\leq k$ for all $y\in Y$, where
$g|f^{-1}(y)$ is the restriction of the map $g$ on $f^{-1}(y)$.

We apply Theorem 1.1 to prove the following its parametric version.

\begin{thm}
Let $f\colon X\to Y$ be a perfect $n$-dimensional map between
metrizable spaces with $\dim Y=0$. Then the following conditions are
satisfied, where $C(X,\mathbb R^m)$ is equipped with the source
limitation topology:
\begin{itemize}
\item[(a)] The set $\mathcal P(3,1,m,3n+1-m)$ is dense and $G_\delta$ in
$C(X,\mathbb R^m)$ provided $m\geq 2n+1$.
\item[(b)] The set $\mathcal P(2,1,m,2n)$ is dense and
$G_\delta$ in $C(X,\mathbb R^m)$ provided $m\geq 2n+1$.
\item[(c)] The set $\mathcal P(1,d,m,n+d(m-d))$ is dense and
$G_\delta$ in $C(X,\mathbb R^m)$ provided $m\geq n+d$.
\item[(d)] The set $\mathcal P(4,1,3,0)$ is dense and
$G_\delta$ in $C(X,\mathbb R^3)$ if $n=1$.
\end{itemize}
\end{thm}

For any map $g\in C(X,\mathbb R^m)$ we also consider the set
$C_{q,d,m}(g)$ consisting of points $y=(y_1,...,y_q)\in(\mathbb
R^m)^q$ such that all $y_i$ belongs to a $d$-plane in $\mathbb R^m$
and there exist $q$ different points $x_i\in X$ with $g(x_i)=y_i$,
$i=1,..,q$. The set of all maps $g\in C(X,\mathbb R^m)$ with $\dim
C_{q,d,m}(g)\leq k$ is denoted by $\mathcal Q(q,d,m,k)$.

Theorem 1.3 below follows from the proof of Theorem 1.2 by
considering the sets $C_{q,d,m}(g)$ instead of $B_{q,d,m}(g)$.

\begin{thm}
Let $X, Y$ and $f$ satisfy the hypotheses of Theorem $1.2$. Then all
items of Theorem $1.2$ remain true if the corresponding sets
$\mathcal P(q,d,m,k)$ are replaced by $\mathcal Q(q,d,m,k)$.
\end{thm}

Theorem 1.3 provides a partial answer of \cite[Question 5]{b2} and
\cite[Conjecture 3]{bv} in the case when $Y$ is a point, $m=T=3$,
$d=1$ and $t=0$.


\section{A preliminary information}
We are going to consider some general statements before proving
Theorem 1.1. Suppose $q\geq 1$ is an integer, $X$ is a metric
compactum.   Let $\Gamma=\{B_1,B_2,..,B_q\}$ be a disjoint family
consisting of $q$ closed subsets of $X$ and $g\in C(X,\mathbb R^m)$.
We denote
$$B_\Gamma(g,m,d)=\{\Pi^d\in M_{m,d}:g^{-1}(\Pi^d)\cap
B_i\neq\varnothing {~}\mbox{for each}{~}i=1,..,q\},$$ where $0\leq
d\leq m$. Now, define the set-valued map
\begin{center}
$\Phi_{\Gamma,m,d}\colon C(X,\mathbb R^m)\to M_{m,d},$
$\Phi_{\Gamma,m,d}(g)=B_\Gamma(g,m,d)$.
\end{center}

\begin{pro}
$\Phi_{\Gamma,m,d}$ is a closed-valued map and
$\Phi_{\Gamma,m,d}^{\sharp}(W)=\{g\in C(X,\mathbb R^m):
\Phi_{\Gamma,m,d}(g)\subset W\}$ is open in $C(X,\mathbb R^m)$ for
every open $W\subset M_{m,d}$.
\end{pro}

\begin{proof}
Suppose $g_0\in\Phi_{\Gamma,m,d}^{\sharp}(W)$ with $W\subset\mathbb
R^m_d$ being open. It suffices to show there exists $\delta>0$ such
that for any $g\in C(X,\mathbb R^m)$ which is $\delta$-close to
$g_0$ we have $B_\Gamma(g,m,d)\subset W$. Assume this is not true.
So, for each $n$ there exists $g_n\in C(X,\mathbb R^m)$ which is
$(1/n)$-close to $g_0$ and $\Pi^d_n\in B_\Gamma(g_n,m,d)$ with
$\Pi^d_n\not\in W$. For any $i\leq q$ and $n\geq 1$ there exists a
point $x_n^i\in B_i\cap g_n^{-1}(\Pi^d_n)$. Since $P=\bigcup_{i\leq
q}g_0(B_i)\subset\mathbb R^m$ is compact, we take a closed ball $K$
in $\mathbb R^m$ with center the origin containing $P$ in its
interior. Because every $\Pi^d\in B_\Gamma(g_0,m,d)$ intersects $P$,
we can identify $B_\Gamma(g_0,m,d)$ with $\{\Pi^d\cap K:\Pi^d\in
B_\Gamma(g_0,m,d)\}$ considered as a subspace of $\mathrm{exp}(K)$
(here $\mathrm{exp}(K)$ is the hyperspace of all compact subset of
$K$ equipped with the Vietoris topology).

Having in mind that for any $x\in X$ the distance in $\mathbb R^m$
between $g_0(x)$ and each $g_n(x)$ is $\leq 1$,  we can assume that
$K$ contains each set $\bigcup_{i\leq q}g_n(B_i)$, $n\geq 1$. Hence,
$g_n(x_n^i)\in K\cap\Pi^d_n$ for all $i\leq q$ and $n\geq 1$.
Therefore, passing to subsequences, we may suppose that there exist
points $x_0^i\in B_i$ $i\leq q$, and a plane $\Pi^d_0\in M_{m,d}$
such that each sequence $\{x_n^i\}_{n\geq 1}$, $i=1,2,..,q$,
converges to $x_0^i\in B_i$ and $\{\Pi^d_n\cap K\}_{n\geq 1}$
converges to $\Pi^d_0\cap K$. So, $\lim\{g_0(x_n^i)\}_{n\geq
1}=g_0(x_0^i)$, $i=1,2,..,q$. Then each $\{g_n(x_n^i)\}_{n\geq 1}$
also converges to $g_0(x_0^i)$. Consequently, $g_0(x_0^i)\in\Pi^d_0$
for all $i$, so $\Pi^d_0\in B_\Gamma(g_0,m,d)$. Hence, $\Pi^d_0\in
W$. Since $W$ is open in $M_{m,d}$ and $\lim\{\Pi^d_n\cap K\}_{n\geq
1}=\Pi^d_0\cap K$ implies that $\{\Pi^d_n\}_{n\geq 1}$ converges to
$\Pi^d_0$ in $M_{m,d}$, $\Pi^d_n\in W$ for almost all $n$, a
contradiction.

The above arguments also show that each $B_\Gamma(g,m,d)$ is closed
in $M_{m,d}$. So, $\Phi_{\Gamma,m,d}$ is a closed-valued map.
\end{proof}

\begin{cor}
Let $X$ and the integers $0\leq d\leq m$ be as in Proposition $2.1$.
Then $\mathcal H(q,d,m,n)$ is a $G_\delta$-subset of $C(X,\mathbb
R^m)$ for any $n\geq 0$ and $q\geq 1$.
\end{cor}

\begin{proof}
We choose a countable family $\mathcal B$ of closed subsets of $X$
such that the interiors of the elements of $\mathcal B$ form a base
for the topology of $X$. Let $\epsilon>0$ and $\Gamma$ be a disjoint
family of $q$ elements of $\mathcal B$. Denote by $\mathcal H_\Gamma
(q,d,m,n,\epsilon)$ the set of all maps $g\colon X\to\mathbb R^m$
such that $B_\Gamma(g,m,d)$ can be covered by an open in $M_{m,d}$
family $\omega$ satisfying the following conditions:
\begin{itemize}
\item[(1)] $\mathrm{mesh}(\omega)<\epsilon$;
\item[(2)] the order of $\omega$ is $\leq n$ (i.e., each point from $M_{m,d}$
is contained in at most $n+1$ elements of $\omega$).
\end{itemize}
Let $g_0\in\mathcal H_\Gamma (q,d,m,n,\epsilon)$ and
$W=\bigcup\{U:U\in\omega\}$. Then $B_\Gamma(g_0,m,d)\subset W$.
According to Proposition 2.1, the set $G=\{g\in C(X,\mathbb
R^m):B_\Gamma(g,m,d)\subset W\}$ is open in $C(X,\mathbb R^m)$, it
contains $g_0$ and $G$ contains $\mathcal H_\Gamma
(q,d,m,n,\epsilon)$. Hence, each $\mathcal
H_\Gamma(q,d,m,n,\epsilon)$ is also open in $C(X,\mathbb R^m)$.

We claim that $$\displaystyle\bigcap\{\mathcal H_\Gamma
(q,d,m,n,1/k):k\geq 1{~}\mbox{and}{~}\Gamma\in\mathcal
B(q)\}=\mathcal H(q,d,m,n),\leqno{(3)}$$ where $\mathcal B(q)$ is
the family of all disjoint subsets of $\mathcal B$ having $q$
elements. Indeed, according to Proposition 2.1, each
$B_\Gamma(g,m,d)$ is a closed subset of $M_{m,d}$. Moreover,
$\displaystyle
B_{q,d,m}(g)=\bigcup\{B_\Gamma(g,m,d):\Gamma\in\mathcal B(q)\}$. So,
by the countable sum theorem for $\dim$, we have $\dim
B_{q,d,m}(g)\leq n$ if and only if $\dim B_\Gamma(g,m,d)\leq n$ for
every $\Gamma\in\mathcal B(q)$. This easily implies equality (3).
Therefore, $\mathcal H(q,d,m,n)$ is $G_\delta$ in $C(X,\mathbb
R^m)$.
\end{proof}

\section{Grassmann manifolds and general position of points and planes}
Let $V^m$ be a vector space of dimension $m$. The Grassmann manifold
$G_{V^m,d}$ (briefly, $\displaystyle G_{m,d}$) is the set of all
$d$-dimensional (vector) subspaces $V^d$ of $V^m$. It is well known
that $\displaystyle G_{m,d}$ has the structure of a smooth compact
manifold, which can be identified with the quotient space
$\displaystyle O(m)/\big(O(d)\times O(m-d)\big)$ (for example, see
\cite[Part II, Chapter 1]{dnf}). Here, $O(m)$ is the orthogonal
group of degree $m$. Since $\dim O(k)=k(k-1)/2$, this implies
$\displaystyle \dim G_{m,d}=(m-d)d$. Suppose $\displaystyle
V^{n_i}_i$ and $0\leq r_i\leq n_i$, $i=1,2,..,k$, are fixed
subspaces of $V^m$ and integers, respectively. Then we denote
$$\displaystyle G_{V^m,d;V^{n_1}_1,r_1;...;V^{n_k}_k,r_k}=\{V^d\in G_{m,d}:\dim
V^d\cap V^{n_i}_i=r_i, i=1,2,..,k\}.$$ Sometimes, instead of
$\displaystyle G_{V^m,d;V^{n_1}_1,r_1;...;V^{n_k}_k,r_k}$ we use the
simpler notation $\displaystyle
G_{m,d;V^{n_1}_1,r_1;...;V^{n_k}_k,r_k}$. If $\displaystyle \dim
(V^{n_1}_1+...+V^{n_k}_k)=n_1+...+n_k$, we write $\displaystyle
G_{m,d;n_1,r_1;...;n_k,r_k}$ instead of $\displaystyle
G_{m,d;V^{n_1}_1,r_1;...;V^{n_k}_k,r_k}$. We have $\displaystyle
G_{m,d;n,r}\neq\varnothing$ if and only if $$0\leq r\leq d\leq
n+d-r\leq m.\leqno{(4)}$$ We are going to consider also the sets
$$\displaystyle G_{m,d;n_1,\geq r_1;...;n_k,\geq r_k}=\bigcup\{G_{m,d;n_1,r_1';...;n_k,r_k'}:r_i'\geq r_i, i=1,...,k\},$$
which are closed in $G_{m,d}$. Since $\displaystyle G_{m,d;n,\geq
r}=G_{m,d;n,\geq r+1}\cup G_{m,d;n,r}$, we have
$$\displaystyle \dim G_{m,d;n,\geq r}=\max\{\dim G_{m,d;n,\geq
r+1},\dim G_{m,d;n,r}\}.\leqno{(5)}$$

Recall that $M_{m,d}$ stands the set of all $d$-planes
$\displaystyle \Pi^d\subset\mathbb R^m$. If $\Pi^{n_i}_i$ and
$-1\leq r_i\leq n_i$, $i=1,2,..,k$, are fixed planes and integers,
we denote
$$\displaystyle M_{m,d;\Pi^{n_1}_1,r_1;...;\Pi^{n_k}_k,r_k}=\{\Pi^d\in M_{m,d}:\dim
\Pi^d\cap\Pi^{n_i}_i=r_i, i=1,2,..,k\}.$$ Identifying every
$d$-plane in $\mathbb R^m$ with a $(d+1)$-dimensional subspace in
$\mathbb R^{m+1}$, we obtain the inclusion
$$\displaystyle M_{m,d;\Pi^{n_1}_1,r_1;...;\Pi^{n_k}_k,r_k}\subset
G_{m+1,d+1;V^{n_1+1}_1,r_1+1;...;V^{n_k+1}_k,r_k+1},$$ where
$V^{n_i+1}_i$ is the subspace of $\mathbb R^{m+1}$ corresponding to
$\Pi^{n_i}_i$. Therefore,
$$\displaystyle \dim M_{m,d;\Pi^{n_1}_1,r_1;...;\Pi^{n_k}_k,r_k}\leq\dim
G_{m+1,d+1;V^{n_1+1}_1,r_1+1;...;V^{n_k+1}_k,r_k+1}.$$

$\Pi(S)$, where $S$ is a subset of $\mathbb R^m$, denotes the affine
hull of $S$, i.e., the smallest affine subspace of $\mathbb R^m$
containing $S$. We say that the planes $\Pi^{n_i}_i$, $i=1,...,k$,
are {\em jointly skew} provided
$$\displaystyle \dim\Pi\big(\Pi^{n_1}_1\cup\Pi^{n_2}_2...\cup\Pi^{n_k}_k\big)=n_1+...+n_k+k-1.$$
In such a case we use the notation $\displaystyle
M_{m,d;n_1,r_1;...;n_k,r_k}$ instead of the general one
$\displaystyle M_{m,d;\Pi^{n_1}_1,r_1;...;\Pi^{n_k}_k,r_k}$.

\begin{pro}
If the integers $m,d,n,r$ satisfy the inequalities $(4)$, then
$\displaystyle \dim G_{m,d;n,r}=(n-r)r+(m-d)(d-r)$.
\end{pro}

\begin{proof}
Let $r=d$. In this case $G_{m,d;n,r}$ consists of all
$r$-dimensional subspaces of $V^n$, i.e., $\displaystyle
G_{m,d;n,r}=G_{n,d}$. Hence, $\displaystyle \dim
G_{m,d;n,r}=(n-r)r$.

If $r=0$, then $n+d\leq m$ and $\displaystyle G_{m,d;n,0}$ is a
non-empty open subset of $G_{m,d}$. Consequently, $\dim
G_{m,d;n,0}=(m-d)d$.

Suppose $0<r<d$ and consider the map $\displaystyle
\varphi:G_{m,d;V^n,r}\to G_{V^n,r}$, $\varphi(V^d)=V^d\cap V^n$.
This map is a locally trivial bundle whose fibre is the space
$\displaystyle
\varphi^{-1}(V^r)=G_{(V^r)^\bot,d-r;V^n\cap(V^r)^\bot,0}=G_{m-r,d-r;n-r,0}$,
where $(V^r)^\bot$ is the orthogonal complement of $V^r$ in $V^m$.
Therefore, according to the previous two cases, $\displaystyle\dim
G_{m,d;n,r}=\dim G_{n,r}+\dim G_{m-r,d-r;n-r,0}=(n-r)r+(m-d)(d-r)$.
\end{proof}

\begin{pro}
If the integers $m,d,n,r$ satisfy the inequalities $(4)$, then
$\displaystyle \dim G_{m,d;n,\geq r}=(n-r)r+(m-d)(d-r)$.
\end{pro}

\begin{proof}
Indeed, according to (4) and Proposition 3.1, $\displaystyle \dim
G_{m,d;n,r+1}<\dim G_{m,d;n,r}$. Then the required inequality
follows from (5).
\end{proof}

\begin{cor}
If the integers $m,d,n,r$ satisfy the inequalities $-1\leq r\leq
d\leq n+d-r\leq m$, then $\displaystyle \dim M_{m,d;n,\geq r}\leq
(n-r)(r+1)+(m-d)(d-r)$.
\end{cor}

\begin{proof}
This follows from Proposition 3.2 and the inclusion $\displaystyle
M_{m,d;n,\geq r}\subset G_{m+1,d+1;n+1,\geq r+1}$.
\end{proof}

\begin{pro}
If the integers $m,n_1,r_1,n_2,r_2$ satisfy the equalities  $0\leq
r_1\leq n_1$, $0\leq r_2\leq n_2$ and $n_1+n_2\leq m$, then
$\displaystyle \dim
G_{m,r_1+r_2;n_1,r_1;n_2,r_2}=(n_1-r_1)r_1+(n_2-r_2)r_2$.
\end{pro}

\begin{proof}
Define the map $\displaystyle\varphi\colon
G_{V^m,r_1+r_2;V^{n_1}_1,r_1;V^{n_2}_2,r_2}\to
G_{V^{n_1}_1,r_1}\times G_{V^{n_2}_2,r_2}$,
$\varphi(V^{r_1+r_2})=\big(V^{r_1+r_2}\cap V^{n_1}_1,V^{r_1+r_2}\cap
V^{n_2}_2\big)$. This map is bijection, its inverse is the map
$\displaystyle\psi\colon G_{V^{n_1}_1,r_1}\times
G_{V^{n_2}_2,r_2}\to G_{V^m,r_1+r_2;V^{n_1}_1,r_1;V^{n_2}_2,r_2}$
given by $\displaystyle\psi\big(V^{r_1},V^{r_2}\big)=V^{r_1}\oplus
V^{r_2}$. So, $\displaystyle G_{m,r_1+r_2;n_1,r_1;n_2,r_2}$ is
homeomorphic to $\displaystyle G_{n_1,r_1}\times G_{n_2,r_2}$.
Consequently, $\displaystyle \dim G_{m,r_1+r_2;n_1,r_1;n_2,r_2}=\dim
G_{n_1,r_1}+\dim G_{n_2,r_2}=(n_1-r_1)r_1+(n_2-r_2)r_2$.
\end{proof}

\begin{rem}
Observe that the following equality was established:\\
$G_{m,r_1+r_2;n_1,\geq r_1;n_2,\geq
r_2}=G_{m,r_1+r_2;n_1,r_1;n_2,r_2}$.
\end{rem}

Because $\displaystyle M_{m,r_1+r_2+1;n_1,\geq r_1;n_2,\geq
r_2}\subset G_{m+1,r_1+r_2+2;n_1+1,\geq r_1+1;n_2+1,\geq r_2+1}$,
Proposition 3.4 and Remark 3.5 imply the next corollary.
\begin{cor}
Let $\Pi^{n_1}_1$ and $\Pi^{n_2}_2$ be two skew planes in $\mathbb
R^m$. Then $\dim M_{m,r_1+r_2+1;n_1,\geq r_1;n_2,\geq r_2}\leq
(n_1-r_1)(r_1+1)+(n_2-r_2)(r_2+1)$.
\end{cor}

Let us note that because the planes $\Pi^{n_1}_1$ and $\Pi^{n_2}_2$
from Corollary 3.6 are skew, we have $m\geq n_1+n_2+1$.

\begin{pro}
Suppose the zero is the only common element of any two of the
subspaces $V_1^{n_1}, V_2^{n_2}, V^r\subset V^m$. If $V^r$ is
contained in a subspace $V^{2r}\subset V^m$ such that $\dim
V^{2r}\cap V_1^{n_1}=\dim V^{2r}\cap V_2^{n_2}=r$, then $V^r\subset
V_1^{n_1}\oplus V_2^{n_2}$. Moreover, $V^{2r}$ is uniquely
determined by the above conditions.
\end{pro}

\begin{proof}
For every such a space $V^{2r}$ the following inclusions hold:
$$\displaystyle V^{r}\subset V^{2r}=\big(V^{2r}\cap
V_1^{n_1}\big)\oplus\big(V^{2r}\cap V_2^{n_2}\big)\subset
V_1^{n_1}\oplus V_2^{n_2}.$$ Thus, $V^r\subset V_1^{n_1}\oplus
V_2^{n_2}$.

Next, consider the subspaces $W_1=V^{r}\oplus V_1^{n_1}$ and
$W_2=V^{r}\oplus V_2^{n_2}$. Since $V^r$ is in a general position
with respect to each $V_i^{n_i}$, $\dim W_i=r+n_i$, $i=1,2$. Then
$W=W_1\cap W_2$ is the required $(2r)$-dimensional subspace. Indeed,
$V^r\subset W$ and, because $V^r\subset V_1^{n_1}\oplus V_2^{n_2}$,
we have $W_1+W_2\subset V_1^{n_1}\oplus V_2^{n_2}$. Hence, $\dim
W=\dim W_1+\dim W_2-\dim (W_1+W_2)=(r+n_1)+(r+n_2)-(n_1+n_2)=2r$. We
also have that $V_1^{n_1}\cap W=V_1^{n_1}\cap W_2$ and
$V_1^{n_1}+W_2=V_1^{n_1}\oplus V_2^{n_2}$. Consequently,
$\dim\big(V_1^{n_1}\cap W\big)=\dim V_1^{n_1}+\dim
W_2-\dim\big(V_1^{n_1}\cap W_2\big)=n_1+r+n_2-(n_1+n_2)=r$.
Similarly, we can obtain that $\dim\big(V_2^{n_2}\cap W\big)=r$.
\end{proof}

\begin{cor}
Let any two of the planes $\Pi_1^{n_1}, \Pi_2^{n_2}, \Pi^r\subset
\mathbb R^m$ be skew. Then there exists at most one $(2r+1)$-plane
$\Pi^{2r+1}\subset\mathbb R^m$ containing $\Pi^r$ such that
$\dim\big(\Pi^{2r+1}\cap\Pi_i^{n_i}\big)\geq r$ for each $i=1,2$.
\end{cor}

\begin{pro}
Suppose the intersection of any two of the subspaces $V_1^{n_1},
V_2^{n_2}, V^{n_3}_3\subset V^m$ is the zero vector. If
$V_1^{n_1}+V_2^{n_2}+V^{n_3}_3=V^m$ and $0\leq r\leq n_1+n_2+n_3-m$,
then $\displaystyle\dim G_{m,2r;V^{n_1}_1,\geq r;V^{n_2}_2,\geq
r;V^{n_3}_3,\geq r}=(n_1+n_2+n_3-m-r)r$.
\end{pro}

\begin{proof}
According to Remark 3.5, the set $\displaystyle
G_{m,2r;V^{n_1}_1,\geq r;V^{n_2}_2,\geq r;V^{n_3}_3,\geq r}$
coincide with $\displaystyle
G_{m,2r;V^{n_1}_1,r;V^{n_2}_2,r;V^{n_3}_3,r}$. So, we need to find
the dimension of the last set. Let $W=\big(V_1^{n_1}\oplus
V_2^{n_2}\big)\cap V_3^{n_3}$. Then $\displaystyle\dim
W=\dim\big(V_1^{n_1}\oplus V_2^{n_2}\big)+\dim
V_3^{n_3}-\dim\big(V_1^{n_1}+V_2^{n_2}+V^{n_3}_3\big)=n_1+n_2+n_3-m\geq
r$. By Proposition 3.7, $\displaystyle
G_{m,2r;V^{n_1}_1,r;V^{n_2}_2,r;V^{n_3}_3,r}$ is homeomorphic to
$G_{W,r}$. Therefore, $\displaystyle\dim
G_{m,2r;V^{n_1}_1,r;V^{n_2}_2,r;V^{n_3}_3,r}=\dim
G_{W,r}=(n_1+n_2+n_3-m-r)r$.
\end{proof}

\begin{rem}
We can suppose that Proposition $3.9$ is also true provided
$r>n_1+n_2+n_3-m$ because in this case $\displaystyle
G_{m,2r;V^{n_1}_1,\geq r;V^{n_2}_2,\geq r;V^{n_3}_3,\geq
r}=\varnothing$.
\end{rem}

\begin{cor}
Suppose any two of the planes $\Pi_1^{n_1}, \Pi_2^{n_2},
\Pi^{n_3}_3\subset \mathbb R^m$ are skew. If
$\Pi\big(\Pi_1^{n_1}\cup\Pi_2^{n_2}\cup\Pi^{n_3}_3\big)=\mathbb R^m$
and $m\leq n_1+n_2+n_3+1-r$, then the dimension of the set
$\{\Pi^{2r+1}\subset\mathbb
R^m:\dim\big(\Pi^{2r+1}\cap\Pi^{n_i}_i\big)\geq r, i=1,2,3\}$ is
$\leq (n_1+n_2+n_3+1-m-r)(r+1)$.
\end{cor}

Recall that a real number $v$ is called algebraically dependent on
the real numbers $u_1,..,u_k$ if $v$ satisfies the equation
$p_0(u)+p_1(u)v+...+p_n(u)v^n=0$, where $p_0(u),..,p_n(u)$ are
polynomials in $u_1,..,u_k$ with rational coefficients, not all of
them 0. A finite set of real numbers is {\em algebraically
independent} if none of them depends algebraically on the others.
The idea to use algebraically independent sets for proving general
position theorems was originated by Roberts in \cite{r}. This idea
was also applied by Berkowitz and Roy in \cite{br}. A proof of the
Berkowitz-Roy theorem was provided by Goodsell in \cite[Theorem
A.1]{g2} (see \cite[Corollary 1.2]{bv} for a generalization of the
Berkowitz-Roy theorem and \cite{g1} for another application of this
theorem). Let us note that any finitely many points in an Euclidean
space $\mathbb R^n$ whose set of coordinates is algebraically
independent are in general position.

It is well known (see for example \cite{cr}) that any hyperboloid of
one sheet $\mathrm H$ in $\mathbb R^3$ is {\em doubly ruled}. This
means that through every one of its points there are two distinct
lines that lie on $\mathrm H$. So, there are two families of lines
on $\mathrm H$ (we call them family I and family II) such that any
two lines on $\mathrm H$ are skew iff they belong to the same
family.

\begin{pro}
For any six points $A_i$, $i=1,..,6$, from $\mathbb R^3$ whose set
of coordinates is algebraically independent there exists a
hyperboloid of one sheet $\mathrm H$ such that:
\begin{itemize}
\item[(a)] The lines $\Pi^1_1=A_1A_2$, $\Pi^1_2=A_3A_4$ and $\Pi^1_3=A_5A_6$ lie on $\mathrm H$ and
belong to one family, say family I;
\item[(b)] If a line $\Pi^1\subset\mathbb R^3$ meets each $\Pi^1_i$, $i=1,2,3$, then
$\Pi^1\subset\mathrm H$ and $\Pi^1$ belongs to family II;
\item[(c)] $\mathrm H$ has an equation whose coefficients are
polynomials in the coordinates of $A_i$, $i=1,..,6$, with rational
coefficients.
\end{itemize}
\end{pro}

\begin{proof}
Since the set of all coordinates of $A_i$, $i=1,..,6$,
is algebraically independent, the following conditions hold:
\begin{itemize}
\item[(6)] there is no 2-dimensional plane in $\mathbb R^3$ containing four of the points $A_i$, $i=1,..,6$;
\item[(7)] there is no 2-dimensional plane in $\mathbb R^3$ parallel to each line $\Pi^1_i$, $i=1,2,3$.
\end{itemize}
Then, according to \cite{cr}, there exists a hyperboloid of one
sheet $\mathrm H$ containing the lines $\Pi^1_i$, $i=1,2,3$. Since
conditions (6) and (7) imply that any two of the lines $\Pi^1_i$,
$i=1,2,3$, are skew, all they belong to one family, say family I.

If a line $\Pi^1\subset\mathbb R^3$ meets each $\Pi^1_i$, $i=1,2,3$,
then $\Pi^1$ has three common points with $\mathrm H$. So,
$\Pi^1\subset\mathrm H$. Moreover, $\Pi^1$ belongs to family II
because each $\Pi^1_i$ belongs to family I.

To prove item (c), observe that the general equation of the
quadratic surface $\mathrm H$ has 10 coefficients $a_{j}$,
$j=1,..,10$. Since $A_i\in\mathrm H$, for each $i=1,2,..,6$ we
obtain a linear with respect to $a_{j}$ equation with coefficients
$c_j^i$, $1\leq j\leq 10$, such that any $c_j^i$ is a polynomial in
the coordinates of $A_i$ with coefficients $1$ or $-1$. Moreover,
$\Pi^1_i\subset\mathrm H$, $i=1,2,3$, yields that each of the
vectors $\overrightarrow{A_1A_2}$, $\overrightarrow{A_3A_4}$ and
$\overrightarrow{A_5A_6}$ has an asymptotic direction. In this way,
there are another three linear with respect to $a_{j}$ equations
whose coefficients are polynomials in the coordinates of $A_i$,
$i=1,2,..,6$ with rational coefficients. So, we have a linear system
of nine equations with respect to $a_{j}$, $j=1,..,10$. The system
has a unique (up to proportions) a non-zero solution. According to
Cramer's rule, this solution can be expressed by rational functions
of the coefficients of the equations. Finally, the proof of (c)
follows from the fact that each of the system's coefficients are
polynomials in the coordinates of $A_i$, $i=1,2,..,6$, with rational
coefficients.
\end{proof}

\begin{cor} Let $\{A_1,...,A_8\}$ be eight points in $\mathbb R^3$ whose set of coordinates
is algebraically independent. Then at most two lines in $\mathbb
R^3$ meets each of the segments $[A_1,A_2]$, $[A_3,A_4]$,
$[A_5,A_6]$ and $[A_7,A_8]$.
\end{cor}

\begin{proof}
Consider a hyperboloid of one sheet $\mathrm H$ satisfying
Proposition 3.12. Suppose a line $\Pi^1\subset\mathbb R^3$ meets
each segment $[A_1A_2]$, $[A_3A_4]$ and $[A_5A_6]$. Then
$\Pi^1\subset\mathrm H$ and $\Pi^1$ belongs to family II. Since the
equation of $\mathrm H$ has coefficients which are polynomials with
rational coefficients in the coordinates of the points $A_i$,
$i=1,..,6$, the coordinates of $A_7$ and $A_8$ don't satisfy the
equation of $\mathrm H$. So, both $A_7$ and $A_8$ are outside
$\mathrm H$. Then the line $A_7A_8$ has at most two common points
with $\mathrm H$. Because there exists exactly one line from family
II passing through a given point of $\mathrm H$, we can have at most
two lines from family II meeting the line $A_7A_8$. This complete
the proof of Corollary 3.13.
\end{proof}

\section{Proof of Theorem 1.1}

We are going first to prove Theorem 1.1(a). In this case we have to
show that the set $\mathcal H(3,1,m,3n+1-m)$ of all maps $g\in
C(X,\mathbb R^m)$ such that $\dim B_{3,1,m}(g)\leq 3n+1-m$ is dense
and $G_\delta$ in $C(X,\mathbb R^3)$. Fix a countable family
$\mathcal B$ of closed subsets of $X$ such that the interiors of its
elements is a base for $X$. Since $\mathcal H(3,1,m,3n+1-m)$ is the
intersection of the open family $\{\mathcal H_\Gamma
(3,1,m,3n+1-m,1/k): k\geq 1\}$ (see the proof of Corollary 2.2), it
suffices to show that each $\mathcal H_\Gamma
(3,1,m,3n+1-m,\epsilon)$ is dense in $C(X,\mathbb R^3)$ (recall that
if $\Gamma=\{B_1,B_2,B_3\}\subset\mathcal B$ is a disjoint family of
three elements, then $\mathcal H_\Gamma (3,1,m,3n+1-m,\epsilon)$
consists of all maps $g\in C(X,\mathbb R^m)$ such that
$B_\Gamma(g,m,1)$ can be covered by an open in $M_{m,1}$ family
$\omega$ with $\mathrm{mesh}(\omega)<\epsilon$) and order $\leq
3n+1-m$).

To this end, observe that each map $g\in C(X,\mathbb R^m)$ can be
approximated by maps $f=h\circ p$ with $p\colon X\to K$ and $h\colon
K\to\mathbb R^m$, where $K$ is a finite polyhedron of dimension
$\leq n$. Actually, $K$ can be supposed to be a nerve of a finite
open cover $\beta$ of $X$. Moreover, if we choose $\beta$ such that
any its element meets at most one element of $\Gamma$, then we have
$p(B_i)\cap p(B_j)=\varnothing$ for $i\neq j$. Further, taking
sufficiently small barycentric subdivision of $K$, we can find
disjoint subpolyhedra $K_i$ of $K$ with $p(B_i)\subset K_i$,
$i=1,2,3$. Obviously, $B_\Gamma(h\circ p,m,1)$ is contained in the
set $B_\Lambda(h,m,1)=\{\Pi^1\in M_{m,1}:h^{-1}(\Pi^1)\cap
K_i\neq\varnothing, i=1,2,3\}$, where $\Lambda$ is the family
$\{K_1,K_2,K_3\}$. Therefore, the density of $\mathcal H_\Gamma
(3,1,m,3n+1-m,\epsilon)$ in $C(X,\mathbb R^m)$ is reduced to show
that the maps $h\in C(K,\mathbb R^m)$ such that $B_\Lambda(h,m,1)$
is covered by an open family in $M_{m,1}$ of
$\mathrm{mesh}<\epsilon$ and order $\leq 3n+1-m$ form a dense subset
of $C(K,\mathbb R^m)$. And this follows from the next proposition.

\begin{pro}
Let $K_i$, $i=1,2,3$, be disjoint at most $n$-dimensional
subpolyhedra of a finite polyhedron $K$ and $m\geq 2n+1$. Then the
maps $h\in C(K,\mathbb R^m)$ such that $\dim B_\Lambda(h,m,1)\leq
3n+1-m$ form a dense subset of $C(K,\mathbb R^m)$, where
$\Lambda=\{K_1,K_2,K_3\}$.
\end{pro}

\begin{proof}
Let $h_0\in C(K,\mathbb R^m)$ and $\delta>0$. We take a subdivision
of $K$ such that $\mathrm{diam}h_0(\sigma)<\delta/2$ for all
simplexes $\sigma$. Let $K^{(0)}=\{a_1,a_2,...,a_k\}$ be the
vertexes of $K$ and $v_j=h_0(a_j)$, $j=1,..,k$. Then, by \cite{br},
there are points $b_j\in\mathbb R^m$ such that the distance between
$v_j$ and $b_j$ is $<\delta/2$ for each $j$ and the coordinates of
all $b_j$, $j=1,..,k$, form an algebraically independent set. Define
a map $h\colon K\to\mathbb R^m$ by $h(a_j)=b_j$ and $h$ is linear on
every simplex of $K$. It is easily seen that $h$ is $\delta$-close
to $h_0$. Without loss of generality, we may assume that each $K_i$,
$i=1,2,3$, is a simplex. Since the coordinates of all vertexes $b_j$
form an algebraically independent set, each $h(K_i)$ generates a
plane $\Pi^{n_i}_i\subset\mathbb R^n$ such that $n_i=\dim h(K_i)\leq
n$ and any two of the planes
$\{\Pi^{n_1}_1,\Pi^{n_2}_2,\Pi^{n_3}_3\}$ are skew. Then, by
Corollary 3.11, the set $$A(h)=\{\Pi^1\in
M_{m,1}:\Pi^1\cap\Pi^{n_i}_i\neq\varnothing, i=1,2,3\}$$ is of
dimension $n_1+n_2+n_3+1-m\leq 3n+1-m$. Because
$A(h)=B_\Lambda(h,m,1)$, we have $\dim B_\Lambda(h,m,1)\leq 3n+1-m$.
This completes the proof.
\end{proof}

As above, the proof of the other three items of Theorem 1.1 is
reduced to the proof of the following proposition.

\begin{pro}
Let $K$ be a finite polyhedron. Then we have:
\begin{itemize}
\item[(a)] If $\Lambda=\{K_1,K_2\}$ is a disjoint pair of at most
$n$-dimensional subpolyhedra of $K$ and $m\geq 2n+1$, then the maps
$h\in C(K,\mathbb R^m)$ with $\dim B_\Lambda(h,m,1)\leq 2n$ form a
dense subset of $C(K,\mathbb R^m)$.
\item[(b)] If $\Lambda=\{K_1\}$ and $m\geq n+1$, where $K_1\subset K$ is
a subpolyhedron with $\dim K_1\leq n$, then the maps $h\in
C(K,\mathbb R^m)$ such that $\dim B_\Lambda(h,m,d)\leq n+d(m-d)$
form a dense subset of $C(K,\mathbb R^m)$, where
$B_\Lambda(h,m,d)=\{\Pi^d\in M_{m,d}:h^{-1}(\Pi^d)\cap
K_1\neq\varnothing\}$.
\item[(c)] If $\Lambda=\{K_1,K_2,K_3,K_4\}$ is a disjoint family of at most
$1$-dimensional subpolyhedra of $K$, then the maps $h\in C(K,\mathbb
R^3)$ with $\dim B_\Lambda(h,3,1)\leq 0$ form a dense subset of
$C(K,\mathbb R^3)$.
\end{itemize}
\end{pro}

\begin{proof}
The same arguments as in the proof of Proposition 4.1 can be used.
The only difference is that, instead Corollary 3.11, we apply now
Corollary 3.6 (with $r_1=r_2=0$) for item (a), Corollary 3.3 (with
$r=0$) for item (b) and Corollary 3.13 for item (c), respectively.
\end{proof}

\section{Proof of Theorem $1.2$.}

We fix a metric $\rho$ generating the topology of $X$. Let $d\in
[1,m]$ and $q\geq 1$ be integers, $g\in C(X,\mathbb R^m)$, $y\in Y$
and $\eta>0$.  We define $\displaystyle B_{q,d,m}^\eta(g,y)$ to be
the set of all $\Pi^d\in M_{m,d}$ such that there exist $q$ points
$x^i\in g^{-1}(\Pi^d)\cap f^{-1}(y)$, $i=1,..,q$, with
$\rho(x^i,x^j)\geq\eta$ for all $i\neq j$. Obviously, $\displaystyle
B_{q,d,m}^\eta(g,y)\subset B_{q,d,m}(g|f^{-1}(y))$ and
$\displaystyle
B_{q,d,m}(g|f^{-1}(y))=\bigcup_{k=1}^{\infty}B_{q,d,m}^{1/k}(g,y)$.

\begin{lem}
Each $\displaystyle B_{q,d,m}^\eta(g,y)$ is closed in $\displaystyle
B_{q,d,m}(g|f^{-1}(y))$.
\end{lem}

\begin{proof}
Suppose we have a sequence $\{\Pi^d_k\}_{k\geq 1}\subset
B_{q,d,m}^\eta(g,y)$ converging in $M_{m,d}$ to a plane $\Pi^d_0$.
Then for every $k$ we have a $q$ points $x_k^i\in
g^{-1}(\Pi^d_k)\cap f^{-1}(y)$, $i=1,..,q$, such that
$\rho(x^i_k,x^j_k)\geq\eta$ for $i\neq j$. Since $f^{-1}(y)$ is a
metric compactum, we can suppose that each sequence
$\{x_k^i\}_{k\geq 1}$ converges to a point $x^i_0\in f^{-1}(y)$.
Then $\lim g(x^i_k)=g(x^i_0)\in\Pi^d_0$, $i=1,..,q$. Moreover,
$\rho(x^i_0,x^j_0)\geq\eta$ for all $i\neq j$. Hence,
$\displaystyle\Pi^d_0\in B_{q,d,m}^\eta(g,y)$.
\end{proof}

Next, if $y\in Y$, $\eta, \epsilon>0$ and $1\leq k$ is an integer,
let $\displaystyle\mathcal P_y^\eta(q,k,d,\epsilon)$ be the set of
all maps $g\in C(X,\mathbb R^m)$ such that $\displaystyle
B_{q,d,m}^\eta(g,y)$ can be covered by an open in $M_{m,d}$ family
of order $\leq k$ and $\mathrm{mesh}<\epsilon$.  If $F\subset Y$, we
consider the set $\displaystyle\mathcal
P_F^\eta(q,k,d,\epsilon)=\bigcap_{y\in F}\mathcal
P_y^\eta(q,k,d,\epsilon)$. Obviously the intersection of all
$\displaystyle\mathcal P_Y^\eta(q,k,d,1/s)$, $s\geq 1$, is the set
$$\displaystyle\mathcal P^\eta(q,k,d)=\{g\in C(X,\mathbb R^m):\dim
B_{q,d,m}^\eta(g,y)\leq k{~}\mbox{for all}{~}y\in Y\}.$$ Moreover,
since $\displaystyle
B_{q,d,m}(g|f^{-1}(y))=\bigcup_{s=1}^{\infty}B_{q,d,m}^{1/s}(g,y)$
and each $\displaystyle B_{q,d,m}^{1/s}(g,y)$ is closed in
$\displaystyle B_{q,d,m}(g|f^{-1}(y))$ (see Lemma 5.1), the
countable sum theorem for the dimension $\dim$ yields that
$\displaystyle\bigcap_{s=1}^{\infty}\mathcal P^{1/s}(q,k,d)$
coincides with the set
$$\displaystyle\mathcal P(q,d,m,k)=\{g\in
C(X,\mathbb R^m):\dim B_{q,d,m}(g|f^{-1}(y))\leq k,y\in Y\}.$$
Therefore, $$\displaystyle\mathcal P(q,d,m,k)=\bigcap\{\mathcal
P^{1/s}_Y(q,k,d,1/p):p\geq 1, s\geq 1\}.$$ So, in order to show that
$\mathcal P(q,d,m,k)$ is dense and $G_\delta$ in $C(X,\mathbb R^m)$,
it suffices to show that each $\displaystyle\mathcal
P^\eta_Y(q,k,d,\epsilon)$ is open and dense in $C(X,\mathbb R^m)$.

We are going first to show that any $\displaystyle\mathcal
P^\eta_Y(q,k,d,\epsilon)$ is open in $C(X,\mathbb R^m)$.

\begin{lem}
Let $\displaystyle g_0\in\mathcal P_{y_0}^\eta(q,k,d,\epsilon)$ for
some $y_0\in Y$. Then there exists a neighborhood $V$ of $y_0$ in
$Y$ and $\delta>0$ such that $\displaystyle g\in\mathcal
P_V^\eta(q,k,d,\epsilon)$ for all $g\in C(X,\mathbb R^m)$ such that
the restrictions $g|f^{-1}(V)$ and $g_0|f^{-1}(V)$ are
$\delta$-close.
\end{lem}

\begin{proof}
Since $\displaystyle g_0\in\mathcal P_{y_0}^\eta(q,k,d,\epsilon)$,
there exists an open in $M_{m,d}$ family $\omega$ of order $\leq k$
and $\mathrm{mesh}(\omega)<\epsilon$ which covers $\displaystyle
B_{q,d,m}^\eta(g_0,y_0)$. Let $W=\bigcup\{U:U\in\omega\}$. It
suffices to show that we can find a neighborhood $V$ of $y_0$ in $Y$
and $\delta>0$ such that $\displaystyle B_{q,d,m}^\eta(g,y)\subset
W$ for all $y\in V$ and all maps $g\in C(X,\mathbb R^m)$ such that
the restrictions $g|f^{-1}(V)$ and $g_0|f^{-1}(V)$ are
$\delta$-close. Suppose this is not true. Then, for each $i=1,..,q$
there exist sequences $\{V_s\}_{s\geq 1}$, $\{y_s\}_{s\geq 1}\subset
Y$, $\{g_s\}_{s\geq 1}\subset C(X,\mathbb R^m)$, $\{\Pi^d_s\}_{s\geq
1}\subset M_{m,d}$ and $\{x_s^i\}_{s\geq 1}\subset X$ satisfying the
following conditions for every $s\geq 1$:
\begin{itemize}
\item $\{V_s\}_{s\geq 1}$ is a local base of neighborhoods at $y_0$;
\item $y_s\in V_s$;
\item $\displaystyle g_s|f^{-1}(V_s)$ and $\displaystyle g_0|f^{-1}(V_s)$ are $(1/s)$-close;
\item $\displaystyle\Pi^d_s\in B_{q,d,m}^\eta(g_s,y_s)\backslash W$;
\item $\displaystyle x_s^i\in g_s^{-1}(\Pi^d_s)\cap f^{-1}(y_s)$ for all $i$;
\item $\displaystyle \rho(x_s^i,x_s^j)\geq\eta$ for all $i\neq j$.
\end{itemize}
As in Proposition 2.1, we can suppose that there exist points
$x^i_0\in f^{-1}(y_0)$, $i=1,..,q$, and a plane $\Pi^d_0\in M_{m,d}$
such that $\lim x_s^i=x_0^i$ and $\lim \Pi^d_s=\Pi^d_0$. It is
easily seen that $\displaystyle\Pi^d_0\in B_{q,d,m}^\eta(g_0,y_0)$
which implies $\Pi^d_0\in W$. This is a contradiction because $\lim
\Pi^d_s=\Pi^d_0$ and $\Pi^d_s\not\in W$.
\end{proof}

\begin{pro}
Any $\displaystyle\mathcal P^\eta_Y(q,k,d,\epsilon)$ is open in
$C(X,\mathbb R^m)$ with respect to the source limitation topology.
\end{pro}

\begin{proof}
We follow the arguments from the proof of \cite[Proposition 3.3]{bv}
(see also the proof of \cite[Proposition 2.3]{v}). For every $y\in
Y$ there exists a neighborhood $V_y$ and a positive $\delta_y\leq 1$
satisfying Lemma 5.2. We suppose that the family $\{V_y:y\in Y\}$ is
locally finite. Define a lower semi-continuous convex-valued map
$\varphi\colon Y\to (0,1]$ by $\varphi(y)=\bigcup\{(0,\delta_z]:y\in
V_z\}$. According to \cite[Theorem 6.2, p.116]{rs}, $\varphi$ admits
a continuous selection $\beta\colon Y\to (0,1]$. Let
$\alpha=\beta\circ f$. Using the choice of the neighborhoods $V_y$
it is easily seen that if $\rho_m(g(x),g_0(x))<\alpha(x)$ for all
$x\in X$, where $g\in C(X,\mathbb R^m)$ and $\rho_m$ is the
Euclidean metric on $\mathbb R^m$, then $\displaystyle g\in\mathcal
P^\eta_Y(q,k,d,\epsilon)$. Therefore, $\displaystyle\mathcal
P^\eta_Y(q,k,d,\epsilon)$ is open in $C(X,\mathbb R^m)$.
\end{proof}

\begin{pro}
Suppose $X,Y$ and $f$ satisfy the hypotheses of Theorem $1.2$ and
$\eta,\epsilon>0$. Then the following holds:
\begin{itemize}
\item[(a)] Each of the sets $\displaystyle\mathcal
P^\eta_Y(3,3n+1-m,1,\epsilon)$ and $\displaystyle\mathcal
P^\eta_Y(2,2n,1,\epsilon)$ is dense in $C(X,\mathbb R^m)$ provided
$m\geq 2n+1$.
\item[(b)] The set $\displaystyle\mathcal
P^\eta_Y(1,n+d(m-d),d,\epsilon)$ is dense in $C(X,\mathbb R^m)$
provided $m\geq n+d$.
\item[(c)] The set $\displaystyle\mathcal
P^\eta_Y(4,0,1,\epsilon)$ is dense in $C(X,\mathbb R^3)$ if $n=1$.
\end{itemize}
\end{pro}

\begin{proof}
We are going to show only that $\displaystyle\mathcal
P^\eta_Y(3,3n+1-m,1,\epsilon)$ is dense in $C(X,\mathbb R^m)$ if
$m\geq 2n+1$, the other proofs are similar. Let $g\in C(X,\mathbb
R^m)$ and $\delta\in C(X,(0,1])$. We are going to find
$h\in\displaystyle\mathcal P^\eta_Y(q,3n+1-m,1,\epsilon)$ such that
$\rho_m(g(x),h(x))<\delta(x)$ for all $x\in X$. By \cite[Proposition
4]{bv1}, $g$ can be supposed to be simplicially factorizable. This
means that there exists a simplicial complex $D$ and maps $g_D\colon
X\to D$, $g^D\colon D\to M$ with $g=g^D\circ g_D$. Following the
proof of \cite[Proposition 3.4]{bv2}, we can find an open cover
$\mathcal U$ of $X$, simplicial complexes $N, L$ and maps
$\alpha:X\to N$, $\beta:Y\to L$, $p\colon N\to L$, $\varphi\colon
N\to\mathbb R^m$ and $\delta_1\colon N\to (0,1]$ satisfying the
following conditions, where $h'=\varphi\circ\alpha$:

\begin{itemize}
\item $\alpha$ is an $\mathcal U$-map and for any $x_1,x_2\in X$ with $\rho(x_1,x_2)\geq\eta$ we have
$\alpha(x_1)\neq\alpha(x_2)$;
\item $\beta\circ f=p\circ\alpha$;
\item $p$ is a perfect $PL$-map with $\dim p\leq n$ and $\dim L=0$;
\item $h'$ is $(\delta/2)$-close to $g$;
\item $\delta_1\circ\alpha\leq\delta$.
\end{itemize}
 So, we have the following commutative diagram:

 \begin{picture}(120,95)(-100,0)
\put(30,10){$L$} \put(0,30){$Y$} \put(12,28){\vector(3,-2){18}}
\put(14,14){\small $\beta$} \put(1,70){$X$}
\put(5,66){\vector(0,-1){25}} \put(-1,53){\small $f$}
\put(11,73){\vector(1,0){45}} \put(30,77){\small $h'$}
\put(12,68){\vector(3,-2){18}} \put(15,56){\small $\alpha$}
\put(31,50){$N$} \put(35,46){\vector(0,-1){25}}
 \put(37,33){\small $p$}
\put(46,58){\vector(4,3){13}} \put(44,64){\small $\varphi$}
 \put(60,70){$\mathbb R^m$}
\end{picture}

Since $L$ is a $0$-dimensional simplicial complex and $p$ is a
perfect $PL$-map, $N$ is a discrete union of the finite complexes
$K_z=p^{-1}(z)$, $z\in L$. Because $\dim p\leq n$, $\dim K_z\leq n$,
$z\in L$. Applying Theorem 1.1(a) to each complex $K_z$, we can find
a map $\varphi_1\colon N\to\mathbb R^m$ such that $\displaystyle\dim
B_{q,1,m}(\varphi_1|p^{-1}(z))\leq 3n+1-m$ and $\varphi_1|p^{-1}(z)$
is $\theta_z$-close to $\varphi|p^{-1}(z)$, where
$\theta_z=\min\{\delta_1(u):u\in p^{-1}(z)\}$. Moreover, the map
$h=\varphi_1\circ\alpha$ is $\delta$-close to $g$. We claim that
$h\in\displaystyle\mathcal P^\eta_Y(q,3n+1-m,1,\epsilon)$. Indeed,
let $y\in Y$ and $\Pi^1\in\displaystyle B_{q,1,m}^\eta(h,y)$. Then
there exist $q$ points $x^i\in h^{-1}(\Pi^1)\cap f^{-1}(y)$,
$i=1,..,q$, with $\rho(x^i,x^j)\geq\eta$ for all $i\neq j$.
According to the choice of the cover $\mathcal U$, we have
$\alpha(x^i)\neq\alpha(x^j)$ for $i\neq j$. Since
$\varphi_1^{-1}(\Pi^1)\cap p^{-1}(\beta(y))$ contains the points
$\alpha(x^i)$, $i\leq q$, we obtain that $\displaystyle\Pi^1\in
B_{q,1,m}(\varphi_1|p^{-1}(\beta(y)))$. Thus, we established the
inclusion $\displaystyle B_{q,1,m}^\eta(h,y)\subset
B_{q,1,m}(\varphi_1|p^{-1}(\beta(y)))$ which implies $\displaystyle
\dim B_{q,1,m}^\eta(h,y)\leq 3n+1-m$ for every $y\in Y$.
Consequently, $h\in\displaystyle\mathcal
P^\eta_Y(q,3n+1-m,1,\epsilon)$.
\end{proof}

\textbf{Acknowledgments.} The results from this paper were obtained
during the second author's visit of Department of Computer Science
and Mathematics (COMA), Nipissing University in August 2010. He
acknowledges COMA for the support and hospitality.


\end{document}